\documentclass[final]{siamart171218}
\usepackage{amsfonts}
\usepackage{amsmath}
\usepackage{amsfonts}
\usepackage{latexsym}
\usepackage{amssymb}
\usepackage{dsfont}
\usepackage[caption=false]{subfig}
\newsiamthm{rem}{Remark}
\newsiamthm{exa}{Example}
\newsiamremark{algo}{Algorithm}
\newsiamthm{nota}{Setting}

\def\coma{\, , \, }
\def\py{\peso{and}}
\newcommand{\peso}[1]{ \quad \text{ #1 } \quad }

\def\n0{n_{ \text{\rm \tiny o}}}

\def\M{\mathbb {M}}
\def\suml{\sum\limits}

\def\bce{\begin{center}}
\def\ece{\end{center}}

\def\py{\peso{and}}

\def\noi{\noindent}
\def\cF{\mathcal F}
\def\cG{\mathcal G}
\def\QED{\hfill $\square$}
\def\EOE{\hfill $\triangle$}

\def\uno{\mathds{1}}
\def\bm{\left(\begin{array}}
\def\em{\end{array}\right)}
\def\ben{\begin{enumerate}}
\def\een{\end{enumerate}}
\def\bit{\begin{itemize}}
\def\eit{\end{itemize}}
\def\barr{\begin{array}}
\def\earr{\end{array}}
\def\igdef{\ \stackrel{\mbox{\tiny{def}}}{=}\ }

\def\la{\lambda}

\def\N{\mathbb{N}}
\def\R{\mathbb{R}}
\def\C{\mathbb{C}}
\def\I{\mathbb{I}}

\def\G{\mathcal{G}}
\def\cH{\mathcal{H}}

\def\cP{\mathcal{P}}

\def\cM{{\cal M}}

\def\cV{{\cal V}}
\def\cU{{\cal U}}

\def\cX{\mathcal{X}}
\def\cY{\mathcal{Y}}
\def\cZ{\mathcal{Z}}

\def\orto{^\perp}

\def\ua{^\uparrow}
\def\da{^\downarrow}

 \DeclareMathOperator{\tr}{tr}

\def\H{{\cal H}}

\newcommand{\mat}{\mathcal{M}_d(\mathbb{C})}

\newcommand{\matsad}{\mathcal{H}(d)}

\newcommand{\matud}{\mathcal{U}(d)}

\newcommand{\matpos}{\mat^+}

\def\beq{\begin{equation}}
\def\eeq{\end{equation}}

\def\pausa{\medskip\noi}

\def\Ax2{\,( S_{E(\cF)^\#_\cV})\hat{}_x }

\newcommand{\nui}[1]{N(#1)}
\def\spr{\text{Spr}}
\def\pyorto{P_{\cY\orto}}
\def\pxorto{P_{\cX\orto}}

\title{Majorization bounds for Ritz values 
\\ of self-adjoint matrices
\thanks{Revised version submitted to the editors on September, 2019.
\funding{Partially supported by CONICET
(PICT ANPCyT 1505/15) and  Universidad Nacional de La Plata (UNLP 11X829)}}}
\author{Pedro Massey%
\thanks{Centro de Matem\'atica, FCE-UNLP and IAM-CONICET, Buenos Aires, Argentina (\email{massey@mate.unlp.edu.ar},
 \email{demetrio@mate.unlp.edu.ar}, \email{seb4.zarate@gmail.com}).
}
\and
Demetrio Stojanoff%
\footnotemark[2]
\and
Sebasti\'an Z\'arate%
\footnotemark[2]
}
\begin{document}
\maketitle

\begin{abstract}
\noindent 
A priori, a posteriori, and mixed type upper bounds for the absolute change in Ritz 
values of self-adjoint matrices in terms of submajorization relations are obtained. 
Some of our results prove recent conjectures by Knyazev, Argentati, and Zhu, 
which extend several known results for one dimensional subspaces to arbitrary subspaces. 
In addition, we improve Nakatsukasa's version of the $\tan \Theta$ theorem of Davis and Kahan. 
As a consequence, we obtain new quadratic a posteriori bounds for the absolute change in Ritz values.
\end{abstract}


\section{Introduction}

The study of sensitivity of Ritz values of Rayleigh quotients of self-adjoint matrices (i.e. the changes in the 
eigenvalues of compressions of a self-adjoint matrix) is a well established and active research field in applied mathematics
\cite{AKPRitz,BosDr,AKFEM,AKMaj,AKProxy,LiLi,Mathias,Ovt,TeLuLi,ZAK,ZK,Z}. Explicitly, given a $d\times d$ complex self-adjoint matrix $A$ 
and isometries $X,\, Y$ of size $d\times k$, with ranges $\cX$ and $\cY$ respectively, we are interested in computing upper and lower bounds
for 
$$ |\la(\rho(X))-\la(\rho(Y))|=(\,|\la_i(\rho(X))-\la_i(\rho(Y))| \, )_{i\in\I_k}\in \R_{\geq 0}^k$$
where $\rho(X)=X^*A\,X, \, \rho(Y)=Y^*A\,Y$ are $k\times k$ complex self-adjoint matrices known as Rayleigh quotients (RQ) of $A$,
 and $\la(\rho(X)),\,\la(\rho(Y)) \in\R^k$ are the eigenvalues (counting multiplicities and arranged in non-increasing order) also known as Ritz values. 

\pausa
Typically, the bounds for the absolute change in the Ritz values are obtained in terms of the 
residuals $R_X=AX-X\, \rho(X)$ and $R_Y=AY-Y \rho(Y)$ or in terms of the principal angles between subspaces (PABS) denoted by $\Theta(\cX,\cY)\in [0,\pi/2]^k$. 
Upper bounds are classified according to which parameters are used to bound the change in Ritz values (see \cite{ZAK}). Indeed, the {\it a priori} bounds 
are those obtained in terms of PABS; the {\it a posteriori} bounds are those obtained in terms of (singular values of) residuals while the {\it mixed type} bounds are obtained in terms 
of both PABS and residuals. It is worth pointing out that PABS appearing in a priori bounds may not be readily available in practice. On the other hand, a 
posteriori bounds are based on computable singular values of residual matrices. 
Moreover, bounds based on residuals (i.e. both a posteriori and mixed type) are particularly convenient in case one of the spaces, say $\cX$, is $A$-invariant 
(as in this case $R_X=0$), as opposed to (autonomous) a priori bounds.

\pausa
The abstract matrix analysis formulation of the sensitivity problem stated above makes it possible to apply 
this theory in a variety of different research areas such as: graph matching \cite{AKMaj} in terms of spectral analysis of the graphs;
 signal distinction in signal processing, where
Ritz values serve as harmonic signature to differentiate subspaces; finite element methods (FEM) \cite{AKFEM}, 
for approximation of subspaces corresponding to fundamental modes; of course, matrix analysis, e.g. for bounds for eigenvalues 
after matrix additive perturbations. Also, bounds for changes in Ritz values 
play a central role in the analysis of algorithms for simultaneous approximation of eigenvalues based on 
Rayleigh-Ritz methods (see \cite{Parlett,StewSun} and the references therein). By now, the role of submajorization in obtaining bounds for the change of Ritz values 
(recognized in the seminal paper \cite{AKMaj}) 
is well known; this partial pre-order relation is a powerful tool in this context, as bounds in terms of 
submajorization imply a whole family of inequalities with respect to 
unitarily invariant norms and with respect to the class of non-decreasing convex functions (\cite{MaOlAr}).

\pausa
In this work we obtain a priori, a posteriori and mixed type upper bounds for the absolute change in Ritz values 
of self-adjoint matrices in terms of submajorization.  Some of our results prove recent conjectures from \cite{AKFEM,ZAK,ZK} which 
extend several known 
results for one dimensional subspaces to arbitrary subspaces. In addition, we improve Nakatsukasa's version 
of the $\tan \Theta$ theorem \cite{Nakats} of Davis and Kahan \cite{DavKah}. We have included some (rather simple) examples to establish comparisons with previous work (for a detailed exposition of the context, 
previous work, our results and some applications, see Section \ref{sec maj err}). We will consider further applications of the results herein elsewhere.

\pausa 
The paper is organized as follows. In Section \ref{sec prelis} we introduce preliminary 
results in majorization theory and principal angles between subspaces.
In Section \ref{sec maj err} we develop our main results; 
our approach to obtain these results is based on methods from abstract matrix analysis, so
we delay the proofs of some technical results until an appendix section.
Section \ref{sec maj err} is divided in three subsections: in Section \ref{sec 3.1} we prove a mixed type
upper bound for the change of the Ritz values that is conjectured in \cite{ZK} and show that this bound is sharp.  We have also included some comments
with a comparison of our results with previous works and with future applications of the results of this subsection.
In Section \ref{subsec applic}
we establish a link between the results from Section \ref{sec 3.1} and an a priori  upper bound for Ritz values conjectured from \cite{AKFEM}. 
Although the results in this section are not sharp, they can be applied in quite general situations and they 
capture the order of approximation conjectured in \cite{AKFEM}. 
 In Section \ref{sec 3.3.} we revisit Nakatsukasa's version of the $\tan \Theta$ theorem 
of Davis and Kahan and obtain an improved version of this result; 
we include an example that shows that this new version of the $\tan \Theta$ theorem is sharp in cases in which the classical 
result is not. As an application, we obtain improved quadratic a posteriori error bounds for Ritz values. 
The paper 
ends with an Appendix (Section \ref{sec append})
in which we include a detailed background on majorization theory and present the 
proofs of some technical results needed in Section \ref{sec maj err}.

\section{Preliminaries}\label{sec prelis}

Throughout our work we use the following 

\pausa
{\bf Notation and terminology}. We let $\mathcal{M}_{d,k}(\C)$ be the space of complex $d\times k$ matrices 
and write $\mathcal{M}_{d,d}(\C)=\mat$ for the algebra of $d\times d$ complex matrices. 
We denote by $\H(d)\subset \mat$ the 
set of self-adjoint matrices and by $\matpos$, the cone of
positive semi-definite matrices. Also, 
$\cG l(d)\subset \mat$ and $\mathcal{U}(d)$  denote the groups of invertible and unitary matrices respectively, 
and $\G l (d)^+ =\G l(d)\cap \matpos$. On the other hand, given a subspace $\cZ\subset \C^d$, we let $\mathcal L(\cZ)$ denote the 
space of linear operators acting on $\cZ$.

\pausa
For $d\in\N$, let $\I_d=\{1,\ldots,d\}$. 
Given a vector $x\in\C^d$ we denote by $D_x$ the diagonal matrix in $\mat$ whose main diagonal is $x$.
Given $x=(x_i)_{i\in\I_d}\in\R^d$ we denote by $x\da=(x_i\da)_{i\in\I_d}$ the vector obtained by 
rearranging the entries of $x$ in non-increasing order. We also use the notation
$(\R^d)\da=\{x\in\R^d\ :\ x=x\da \}$, $R_{\ge 0} = 
\{x\in \R: x\ge 0\}$ and $(\R_{\geq 0}^d)\da=
\{x\in\R_{\geq 0}^d\ :\ x=x\da \}$. For $r\in\N$, we let $\uno_r=(1,\ldots,1)\in\R^r$.

\pausa
 Given a matrix $A\in\matsad$ we denote by $\la(A)=(\la_i(A))_{i\in\I_d}\in (\R^d)\da$ 
the eigenvalues of $A$ counting multiplicities and arranged in 
non-increasing order.   
For $B\in\mat$ we let $s(B)=\la(|B|)$ denote the singular values of $B$, i.e. the eigenvalues of $|B|=(B^*B)^{1/2}\in\matpos$. We use the abbreviation ONB for ``orthonormal basis". 

\pausa Arithmetic operations with vectors are performed entry-wise i.e., 
in case $x=(x_i)_{i\in\I_k} $ and $ y=(y_i)_{i\in\I_k}\in \C^k $ 
then $x+y=(x_i+y_i)_i$ and,  following the notational convention of the principal references on these matters, 
$$
x\, y
 =(x_i\,y_i)_i \peso{ and (assuming that $y_i\neq 0$, for $i\in\I_k$)} 
\frac xy =(x_i/y_i)_i  \ ,
$$ 
where these vectors all lie in $\C^k$. Moreover, if we assume further that $x,\,y\in\R^k$ then we write $x\leq y$ whenever 
$x_i\leq y_i$, for $i\in\I_k$.
\EOE

\pausa Next we recall the notion of majorization between vectors, that will play a central role throughout our work.
\begin{definition}\rm 
Let $x,\, y\in\R^k$. We say that $x$ is
{\it submajorized} by $y$, and write $x\prec_w y$,  if
$$
\suml_{i=1}^j x^\downarrow _i\leq \suml_{i=1}^j y^\downarrow _i \peso{for} j\in\I_k\,. 
 $$ If $x\prec_w y$ and $\tr x \igdef \suml_{i=1}^kx_i=
 \tr y$,  then we say that $x$ is
{\it majorized} by $y$, and write $x\prec y$. \EOE
\end{definition}

\def\int{[a\coma b]}

\pausa
There are many fundamental results in matrix theory that are stated in terms of submajorization relations.
In what follows, we mention some elementary properties of submajorization that we will need in Section \ref{sec maj err}
(for detailed expositions on majorization theory, including proofs of the results mentioned below, see \cite{bhatia,HJ,MaOlAr}). 
We will consider some further properties and results on majorization theory 
in Section \ref{sec append}.
Given $f: \int \rightarrow \R$, where $\int\subset \R$ is an interval, and $z=(z_i)_{i\in\I_k}\in \int^k$ we denote $f(z)=(f(z_i))_{i\in\I_k}\in\R^k$.
\begin{rem}\label{convfunction}
Let $\int\subset \R$ be an interval and let  
$f: \int \rightarrow \R$ be a convex function. Then, 
\begin{enumerate}
\item if $x,\, y\in \int^k$ satisfy $x\prec y$ then $f(x)\prec_w f(y)$. 
\item If $x,\, y\in \int^k$ only satisfy $x\prec_w y$ but $f$ is further non-decreasing in $\int$, then $f(x)\prec_w f(y)$. 
\end{enumerate}
\EOE
\end{rem}

\begin{definition}\rm 
A norm $N$ in $\mat$ is {\bf unitarily invariant} (briefly u.i.n.) if 
$\nui{UAV}=\nui{A}$, for every $A\in\mat$ and $U,\, V\in\mathcal{U}(d)$.
\EOE
\end{definition}
\pausa
Well known examples of u.i.n. are the spectral norm $\|\cdot\|_{sp}$ and 
the Schatten $p$-norms $\|\cdot\|_p$, for $p\geq 1$. 

\begin{rem}\label{Domkyfan}\rm
It is well known that (sub)majorization relations between singular values of matrices are intimately related 
with inequalities with respect to u.i.n's. Indeed, given $A,\, B\in\mat$ the following statements are equivalent:
\begin{enumerate} 
\item For every u.i.n. $N$ in $\mat$
we have that $N(A)\leq N(B)$.
\item $s(A)\prec_w s(B).$ \EOE
\end{enumerate} 
\end{rem}

\pausa
{\bf Principal Angles Between Subspaces}. Let $\cX,\, \cY\subset \C^d$ denote subspaces, with $\dim\cX=h$ and $\dim \cY=k$.
Let $X\in\cM_{d,h}$ and $Y\in\cM_{d,k}$ be such that their columns form orthonormal bases of $\cX$ and $\cY$ respectively.
Then, the principal angles between $\cX$ and $\cY$, denoted $\pi/2\geq \Theta_1(\cX,\cY)\geq \ldots\geq \Theta_m(\cX,\cY)\geq 0$
where $m=\min\{h,k\}$ - are determined by 
$$\cos(\Theta_{m-i+1}(\cX,\cY))= s_i(X^*Y) \peso{for} i\in\I_m\,.$$
We further write $\Theta(\cX,\cY)=(\Theta_i(\cX,\cY))_{i\in\I_m}\in(\R^m)\da$ for the vector of principal angles between 
$\cX$ and $\cY$. Principal angles are a useful tool in describing the relative position and several geometric and metric aspects related with 
the subspaces $\cX$ and $\cY$ in $\C^d$ 
(see \cite{DavKah,Hal} and the references therein).

\section{Main results}\label{sec maj err}

In this section we develop our main results. 
The section is divided in three parts; 
first we prove \cite[Conjecture 2.1]{ZK} which 
establishes a mixed type bound for the error in the (absolute) change
of the Ritz values.
In the second part, we establish connections between the mixed type bounds of the first section
and some a priori bounds for the change of Ritz values conjectured in \cite{AKFEM,AKProxy}.
Finally we take a closer look at Nakatsukasa's $\tan \Theta$ theorem under relaxed conditions from \cite{Nakats} and 
obtain an improved version of this result. As a consequence we obtain quadratic a posteriori error bounds
for the change of the Ritz values that improve several known bounds.
Our approach to obtain these results is based on methods from abstract matrix analysis,
 so we delay the proofs of some technical results until Section \ref{sec append},  
where we have also included several classical results of this area that we will refer to in this section.

\pausa
We begin by introducing the following
\begin{nota}\label{nota1} \rm
Throughout this section we consider the following notation and terminology: 
\begin{enumerate}
\item $\mathcal X,\,\mathcal Y\subset \C^d$ denote two subspaces of dimension $k$. 
We fix $X,\, Y\in \mathcal M_{d,k}(\C)$ such that their columns form orthonormal 
bases of $\cX$ and $\cY$, respectively. 
\item $\Theta(\cX\coma\cY)\in (\R_{\geq 0}^k)\da$ denotes the vector of principal angles between 
the subspaces $\cX$ and $\cY$; in this case, 
$$\cos(\Theta\ua(\cX\coma\cY))=s(X^*Y)=(s_1(X^*Y),\ldots,s_k(X^*Y)) \in(\R_{\geq 0}^k)\da .$$ 
\item For a (fixed) self-adjoint $A\in\matsad$ we set 
$\rho(X)=X^*AX\in\cM_k(\C)$, $R_X=AX-X\rho(X)\in\cM_{d,k}(\C)$ and similarly $\rho(Y)$ and $R_Y$ for $Y$.
Notice that $$ R_X=AX-XX^*AX=AX-P_\cX AX=P_{\cX^\perp} AX\in\cM_{d,k}(\C)\,,$$
where $P_\cX\in\mat$ denotes the orthogonal projection onto $\cX$ and $\cX^\perp$ denotes the orthogonal complement of $\cX$. We consider similar 
notation and identities for $\cY$. 
\item 
Let $X_\perp\in \cM_{d\coma d-k}(\C)$ be such that its columns form an ONB of $\cX\orto$. 
Then, the matrix $\left(X,X_\perp\right)\in\matud$ and we get 
$$\tilde A=
\left(X,  X_\perp\right)\  A\ \left(X,  X_\perp\right) ^* 
= \bm{cc} \rho(X)&R_X^*\, X_\perp\\X_\perp^*\, R_X &\rho(X_\perp) \em 
\, . 
$$
Note that, since  $R_X = (I-P_\cX)\,R_X\,$, then 
$s(R_X)= s(X_\perp^*\, R_X )$, so that we can think of 
$R_X$ (up to an isometric factor) as the $(2,1)$-block of 
$\tilde A$, 
in the block matrix representation of 
(the unitary conjugate of $A$) $\tilde A$ as above. 
\EOE
%
%
%
\end{enumerate}
\end{nota}

\subsection{Rayleigh-Ritz majorization error bounds of the mixed type}\label{sec 3.1}

We adopt Setting \ref{nota1}; moreover, in this subsection
we further assume that $\cX$ and $\cY$ are such that $\Theta_1(\cX\coma \cY)<\frac{\pi}{2}$ that is, that $X^*Y\in\cG l(k)$ is invertible.

\pausa
Our first result concerns a submajorization error bound for the distance of eigenvalue lists of self-adjoint matrices:

\begin{theorem}\label{lemmaa3'} Let $C,\, D\in\cH(k)$ and let $T\in\cG l(k)$. Then,
\begin{equation}\label{eqn7}
|\la(C)-\la(D)|\prec_w s(T^{-1})\ s(CT-TD).
\end{equation}
\end{theorem}
\begin{proof} See the Appendix (Section \ref{sec append}).
\end{proof}

\pausa The following result is \cite[Conjecture 2.1]{ZK} (see also Corollary \ref{cor caso invariante} below).

\begin{theorem}\label{theoremrC1} Under Setting \ref{nota1}, if 
$\Theta_1(\cX\coma \cY)<\frac{\pi}{2}\,$ then 
\beq\label{theorem1} |\lambda(\rho(X)) - \lambda(\rho(Y))|\prec_w \frac{s(P_\cY\ R_X) + s(P_\cX\ R_Y)}{\cos(\Theta(\cX\coma \cY))} 
\py\eeq
 \beq\label{theorem2} |\lambda(\rho(X)) - \lambda(\rho(Y))|\prec_w [s(P_{\cX+\cY}\  R_X) + s(P_{\cX+\cY}\ R_Y)] \,\tan(\Theta(\cX\coma \cY))\,.\eeq
\end{theorem}

\begin{proof}  
Set $T=X^{\ast}Y$ and notice that, since $\Theta_1(\cX,\cY)<\frac{\pi}{2}$,  $T\in\mathcal{M}_k(\C)$ is invertible. Using Theorem \ref{lemmaa3'} we get that 
\begin{align}\label{eqn9}
|\la(\rho(X))-\la(\rho(Y))| \prec_w s(T^{-1})\ s(\rho(X)T-T\rho(Y))\, ,
\end{align} where $\rho(X)=X^*AX,\, \rho(Y)=Y^*AY\in \cH(k)$. By construction we have that 
\beq\label{eqn10}
 s(T^{-1})=\frac{1}{\cos(\Theta(\cX,\cY))}\in (\R^k_{>0})^{\downarrow}\,.
\eeq
 Arguing as in \cite[Thm 4.1]{ZK} we notice that
\begin{align*}
\rho(X)T-T\rho(Y)&=X^*A\,XX^*Y-X^*YY^*A\,Y=X^*A\,P_\cX Y-X^*P_\cY A\,Y
\\ &=X^*A\, (I-P_{\cX^{\perp}})Y-X^*(I-P_{\cY^{\perp}})A\, Y
\\&=X^*A\,Y-X^*A\,P_{\cX^{\perp}}Y-X^*A\,Y+X^*P_{\cY^{\perp}}A\,Y=-X^*A\,P_{\cX^{\perp}}Y+X^*P_{\cY^{\perp}}A\,Y\,.
\end{align*}
Using that $s(C)=s(C^*)$ for $C\in\mathcal M_k(\C)$, we see that 
$$
s(X^*A\,P_{\cX^{\perp}}Y)=s(Y^*P_{\cX^{\perp}}A\,X)=s(P_\cY P_{\cX^{\perp}}A\,X)=s(P_\cY R_X)\in(\R_{\geq 0}^k)\da\,.
$$
Analogously 
$
s(X^*P_{\cY^{\perp}}A\,Y)=s(P_\cX R_Y)
$. The previous facts together with the sub-additivity property of taking singular values 
(item 1 in  Theorem \ref{theorem ag}) imply that 
\beq\label{eqn11} 
s(\rho(X)T-T\rho(Y))= s(-X^*A\,P_{\cX^{\perp}}Y+X^*P_{\cY^{\perp}}A\,Y)\prec_w  s(P_\cX R_Y)+s(P_\cY R_X)\,.
\eeq
Now, if we apply \eqref{eqn10} and \eqref{eqn11} to  \eqref{eqn9}, 
together with item 4 in Lemma \ref{lemma submaj props1}, 
we get 
\eqref{theorem1}.

\pausa 
In order to show \eqref{theorem2} we point out that by \cite[Lemma 4.1]{ZK} we get that 
\beq \label{eq para C1y medio}
s(P_\cX R_Y ) \prec_w s (P_{\cX+\cY}\, R_Y) \sin (\Theta(\cX, \cY)) \,.
\eeq 
Since the entries of these vectors are ordered downwards, by Lemma \ref{lemma submaj props1} we deduce that  
\beq \label{eq para C2}
s(P_\cX R_Y ) + s(P_\cY R_X) 
 \prec_w  \big(\, s (P_{\cX+\cY}\,R_Y ) 
 + s (P_{\cX+\cY}\, R_X) \, \big) \, \sin (\Theta(\cX, \cY)) \ . 
\eeq
Hence, using \eqref{theorem1} and \eqref{eq para C2} together with Lemma \ref{lemma submaj props1} we see that 
\eqref{theorem2} holds.
\end{proof}

\pausa The fact that \eqref{theorem1} implies \eqref{theorem2} was already observed in \cite{ZK}; we have included 
the proof of this fact for the benefit of the reader.

\begin{corollary}\label{cor caso invariante}
Consider Setting \ref{nota1} and assume that $\Theta_1(\cX\coma \cY)<\frac{\pi}{2}$. If we further assume that 
$\cX$ is $A$-invariant then
\beq\label{eq coro1} |\lambda(\rho(X)) - \lambda(\rho(Y))|\prec_w \frac{s(P_\cX\ R_Y)}{\cos(\Theta(\cX\coma \cY))} 
\quad \text{ and } \eeq
 \beq\label{eq coro2} |\lambda(\rho(X)) - \lambda(\rho(Y))|\prec_w s(P_{\cX+\cY}\ R_Y) \, \tan(\Theta(\cX\coma \cY))\,.\eeq
\end{corollary}
\begin{proof}
In case $\cX$ is $A$-invariant notice that $R_X=0$. The result now follows from Theorem \ref{theoremrC1}.
\end{proof}

\pausa
It is natural to wonder whether we can improve the bounds in the previous results.
As shown in the following example, the submajorization bounds in Theorem \ref{theoremrC1} and Corollary \ref{cor caso invariante} are {\it sharp}.

\begin{exa}\label{exa1}\rm 
Let $\lambda=(a,b,c,d)\in\R^4$, where $a<b<c<d$, and consider 
$A\in\cH(4)$ given by $A=D_\lambda$, i.e. $A$ is the diagonal matrix with main diagonal $\lambda$. 

\pausa
Let $\cX$ be the $A$-invariant subspace 
$\cX=\text{span}\{e_1,\,e_2\}$ spanned by the first two elements of the canonical basis of $\C^4$. For $\theta\in (0,\pi/2)$ let 
$f_\theta=\cos\theta \, e_2 + \sin \theta \, e_3$ and set $\cY_\theta=\text{span}\{e_1,\,f_\theta\}$. Then, 
the principal angles are given by $\Theta(\cX,\cY_{\theta})=(\theta,0)$. 
Let 
$$
X=\begin{pmatrix}
1 & 0 \\
0 & 1 \\
0 & 0 \\
0 & 0 
\end{pmatrix} 
\peso{,}
X_{\perp}=\begin{pmatrix}
0 & 0 \\
0 & 0 \\
1 & 0 \\
0 & 1 
\end{pmatrix} 
\py 
Y_\theta=\begin{pmatrix}
1 & 0 \\
0 & \cos \theta \\
0 & \sin \theta \\
0 & 0 
\end{pmatrix} \,.
$$ 
It is straightforward to check that $\lambda(X^*AX)=(b,a)$ and that $\lambda(Y_\theta^*AY_\theta)=(b\,\cos^2\theta + c\, \sin^2(\theta) , a)$.
Again, simple computations show that 
$$R_{Y_\theta}= 
\begin{pmatrix}
0 & 0 \\
0 & (b-c)\,\cos\theta\, \sin^2\theta \\
0 & (c-b)\,\cos^2\theta\, \sin\theta \\
0 & 0 
\end{pmatrix} \peso{,} P_\cX\, R_{Y_\theta}= 
\begin{pmatrix}
0 & 0 \\
0 & (b-c)\,\cos\theta\, \sin^2\theta \\
0 & 0\\
0 & 0 
\end{pmatrix} \,.
$$ Hence, $s(P_\cX\, R_{Y_\theta})=( (c-b) \cos \theta \, \sin^2\theta, 0)$. Now, 
\beq \label{eq conc theoremC11}
|\lambda(X^*A\,X)-\lambda((Y_\theta)^*A\,Y_\theta)|=((c-b)\,\sin^2\theta,0) \   , 
\eeq
\beq \label{eq conc theoremC12}
 \frac{s(P_\cX\ R_{Y_\theta})}{\cos(\Theta(\cX\coma \cY_\theta))} = 
((c-b)\,\sin^2\theta,0)\,.
\eeq
That is, \eqref{eq coro1} in Corollary \ref{cor caso invariante} becomes an equality in this case. 
This also shows that \eqref{theorem1} is sharp, since \eqref{eq coro1} above is a particular case (when $\cX$ is $A$-invariant). 
%
Notice that $\cX+\cY_\theta=\text{span}\{e_1,e_2,e_3\}$. Therefore, 
since $P_{\cX+\cY_\theta}\,R_{Y_\theta}=R_{Y_\theta}$ and $s(R_{Y_\theta})
=( (c-b) \cos \theta \, \sin\theta, 0)$, 
\beq \label{eq conc theoremC2}
 s(P_{\cX+\cY_\theta}\ R_{Y_\theta}) \, \tan(\Theta(\cX\coma \cY_\theta)) = ( (c-b) \, \sin^2 \theta, 0)\,.
\eeq By \eqref{eq conc theoremC11} and \eqref{eq conc theoremC2} we now see that \eqref{eq coro2}
in Corollary \ref{cor caso invariante} becomes an equality in this case. 
This also shows that \eqref{theorem2} is sharp, since \eqref{eq coro2} above
 is a particular case (when $\cX$ is $A$-invariant). 
\EOE
\end{exa}

\begin{rem}[Relations between our work and previous results]\rm 
In the vector case, that is when $\cX$ and $\cY$ are one dimensional spaces, Theorem \ref{theoremrC1} 
implies the upper bounds in \cite[Theorem 3.7]{ZAK}, which is one of 
the main results of that work (see also Corollary \ref{cor ZAK extend} 
and Remark \ref{rem aplic conj tan2}).

\pausa
In \cite{ZK} Knyazev and Zhu obtained several bounds for the absolute change of the Ritz values. Using Setting \ref{nota1}, the authors show (see \cite[Theorem 4.2 and Corollary 4.4]{ZK}) that 
\beq\label{eq ZK1} 
|\lambda(\rho(X)) - \lambda(\rho(Y))|^2 \prec_w \frac{\{s(P_\cY\ R_X) + s(P_\cX\ R_Y)\}^2 }{\cos^2(\Theta(\cX\coma \cY))} 
\py
\eeq
 \beq\label{eq ZK2} |\lambda(\rho(X)) - \lambda(\rho(Y))|^2 \prec_w \{s(P_{\cX+\cY}\  R_X) + s(P_{\cX+\cY}\ R_Y)\}^2  \,\tan^2(\Theta(\cX\coma \cY))\,.
\eeq
Using the fact that $f:\R_{\geq 0}\rightarrow \R_{\geq 0}$ given by $f(x)=x^2$ 
is an increasing and convex function, Remark \ref{convfunction} shows 
that \eqref{eq ZK1} and \eqref{eq ZK2} follow from \eqref{theorem1} and 
\eqref{theorem2} 
from Theorem \ref{theoremrC1}. Similarly, using that 
$\cos\Theta_1(\cX,\cY)=\cos\Theta_{\max}(\cX,\cY)\leq \cos\Theta_i(\cX,\cY)$, for $i\in\I_k$, we get that 
Theorem \ref{theoremrC1} implies \cite[Theorems 4.1, 4.3]{ZK}. 

\pausa
In \cite{ZK} the authors show that their results can be applied in several situations such as: first order and 
quadratic a posteriori majorization bounds; bounds for eigenvalues after matrix additive perturbations.
The previous remarks show that our bounds can also be applied in these
settings.  Moreover, Theorem \ref{theoremrC1} allows to formalize the arguments related with bounds for eigenvalues 
after matrix additive perturbations, and in particular with bounds for eigenvalues 
after discarding off-diagonal blocks from \cite[Section 5]{ZK} (see the detailed discussion there). 
\EOE
\end{rem}
\pausa
The bounds in Theorem \ref{theoremrC1} can be used to perform a detailed analysis and obtain 
better convergence rates for iterative algorithms related with the Rayleigh-Ritz method (see \cite{Parlett,StewSun,Z}).
We will consider such applications elsewhere.

\subsection{Applications: a priori majorization error bounds for Ritz values}\label{subsec applic}

In this section we establish a link between the majorization error bounds of the mixed type obtained in the previous section 
and some a priori majorization error bounds considered in \cite{AKFEM,AKProxy}.

\begin{definition}\label{spread}
Let $A\in \matsad$ and let $\cZ\subset \C^d$ be a subspace with $\dim \cZ=p$. 
We define the (spectral) spread of $A$ relative to $\cZ$, denoted 
${\rm Spr}(A\coma \cZ)$, given by
$$
{\rm Spr}(A,\cZ)=\lambda(A_\cZ)-\lambda^\uparrow(A_\cZ)
=(\la_i(A_\cZ)-\la_{p-i+1}(A_\cZ))_{i\in\I_p} \in(\R^p)\da\, ,
$$
where $A_\cZ=P_\cZ\, A|_\cZ\in \mathcal L(\cZ)$ is a self-adjoint operator (defined in the obvious way). 
In case $\cZ=\C^d$,  we write 
${\rm Spr}(A,\C^d)={\rm Spr}(A)$.
\EOE 
\end{definition}

\begin{rem}\rm 
Let $A\in \H(d)$ and let $\cX,\cY\subset \C^d$ with $\dim (\cX)=\dim(\cY)=k$. Denote by $p = \dim \cX + \cY$. In what follows we consider the vector 
$$\spr(A,\cX+\cY) \, \sin(\Theta(\cX,\cY))= (\,(\la_i(A_{\cX+\cY})-\la_{p-i+1}(A_{\cX+\cY}) )\, \sin(\Theta_i(\cX,\cY) \,)\,)_{i\in\I_k}\,.$$
We point out that this vector has non-negative entries, which are arranged in non-increasing order (in particular,
$\sin(\Theta_i(\cX,\cY))=0$ whenever $\la_i(A_{\cX+\cY})-\la_{p-i+1}(A_{\cX+\cY})$ $<0$, for $i\in\I_k$); hence, 
$ \spr(A,\cX+\cY) \, \sin(\Theta(\cX,\cY))\in (\R_{\geq 0}^k)\da$ (see \cite{ZK}). This fact becomes relevant for the conjectures posed
in \eqref{eq conj apr1} and \eqref{eq conj apr2} below.
\EOE
\end{rem}

\begin{rem}[A priori error bounds for changes of Ritz values: conjectures and previous work]\label{rem prepa1}\rm 
Let $A\in \H(d)$ and let $\cX,\cY\subset \C^d$ with $\dim (\cX)=\dim(\cY)=k$.
In 	\cite{AKFEM} the authors conjectured that, in general, the following submajorization bound for the Ritz values holds:
\beq\label{eq conj apr1}
|\la(\rho(X))-\la(\rho(Y))|\prec_w \spr(A,\cX+\cY) \, \sin(\Theta(\cX,\cY))\,.
\eeq Moreover, in case $\cX$ is $A$-invariant, the authors conjectured that 
\beq\label{eq conj apr2}
|\la(\rho(X))-\la(\rho(Y))|\prec_w \spr(A,\cX+\cY) \, \sin(\Theta(\cX,\cY))^2\,.
\eeq 
These conjectures are natural extensions of results from \cite{AKProxy} (that were obtained for $k=1$). 
Although \cite[Conjecture 2.1.]{AKFEM} claims the validity of \eqref{eq conj apr1} and \eqref{eq conj apr2} for arbitrary 
subspaces $\cX$ and $\cY$ such that $\dim \cX=\dim \cY$, 
such bounds would become relevant in the particular case when the subspace $\cY$ is a (small) perturbation of the 
subspace $\cX$. In this case, the validity of \eqref{eq conj apr1} and \eqref{eq conj apr2} 
would reveal the different orders of approximation of $\rho(X)$ by $\rho(Y)$ in terms of 
PABS as well as in terms of the spectral spread of $A$ (i.e. when considering $A$ as well as $\cX$ and $\cY$ as variables). Notice that these results would have immediate applications
in the study of numerical stability and convergence of iterative methods related with the Rayleigh-Ritz type algorithms.

\pausa
In \cite[Theorem 2.1.]{AKFEM} the authors showed that, in general, 
\beq\label{eq theorem apr1}
|\la(\rho(X))-\la(\rho(Y))|\prec_w (\la_{\max}(A_{\cX+\cY})- \lambda_{\min}(A_{\cX+\cY})) \, \sin(\Theta(\cX,\cY))\,,
\eeq while, in case $\cX$ is $A$-invariant,
\beq\label{eq theorem apr2}
|\la(\rho(X))-\la(\rho(Y))|\prec_w (\la_{\max}(A_{\cX+\cY})- \lambda_{\min}(A_{\cX+\cY})) \, \sin(\Theta(\cX,\cY))^2\,,
\eeq where $A_{\cX+\cY}=P_{{\cX+\cY}}\ A|_{{\cX+\cY}}\in \mathcal L({\cX+\cY})$; moreover, in \cite[Theorem 2.2.]{AKFEM} they showed that in the particular
case in which $\cX$ is the $A$-invariant subspace corresponding to the $k$ largest eigenvalues of $A$, then 
\beq\label{eq theorem apr3}
0\leq \la(\rho(X))-\la(\rho(Y))\prec_w (\la_i(A_{\cX+\cY})- \lambda_{\min}(A_{\cX+\cY}))_{i\in\I_k}\, \sin(\Theta(\cX,\cY))^2\,.
\eeq Notice that, \eqref{eq theorem apr3} is a stronger bound than that in \eqref{eq theorem apr2}; yet, it is weaker 
than the bound conjectured in \eqref{eq conj apr2}, since
$\spr_i(A,\cX+\cY)\leq \la_i(A_{\cX+\cY})- \lambda_{\min}(A_{\cX+\cY})$, for $i\in\I_k$. 
\EOE
\end{rem}
\pausa
In what follows we apply Theorem \ref{theoremrC1} and obtain some results related with the conjectures 
from \cite{AKFEM} described in \eqref{eq conj apr1} and \eqref{eq conj apr2}. In order to obtain these results, we take a closer look at the quantity $s(P_{\cX}\, R_Y)$ for arbitrary $\cX$ and $\cY$, as well as in the case where $\cX$ is $A$-invariant. 

\begin{proposition}\label{prospread}
Let $A\in \H(d)$ and let $\cX,\cY\subset \C^d$ with $\dim (\cX)=\dim(\cY)=k$. Then
 \begin{equation}\label{eqn12}
s(P_{\cX}\, R_Y)\prec_w {\rm Spr}(A,\cX+\cY) \, \sin(\Theta(\cX,\cY))\,.
\end{equation}
\end{proposition}
\begin{proof} See the Appendix (Section \ref{sec append}). \end{proof}

\begin{theorem}\label{theorem apriori nostro}
Let $A\in \H(d)$, $\cX,\cY\subset\C^d$ subspaces, $\dim(\cX)=\dim(\cY)=k$. 
If $\Theta_1(\cX,\cY)<\frac{\pi}{2}$, then 
\begin{eqnarray}
\label{T1spread} |\lambda(\rho(X)) - \lambda(\rho(Y))|&\prec_w & 
\frac{2\,{\rm Spr}(A,\cX+\cY)\, \sin(\Theta(\cX,\cY))}{\cos(\Theta(\cX\coma \cY))} 
\,.\end{eqnarray}
\end{theorem}

\begin{proof}
Theorem \ref{theoremrC1} establishes that $$|\la(\rho(X))-\la(\rho(Y))|\prec_w \frac{s(P_{\cX}R_{Y})+s(P_{\cY}R_{X})}{\cos(\Theta(\cX,\cY))}\,.$$
Proposition \eqref{prospread} together with Lemma \ref{lemma submaj props1} imply that
$$ 
\frac{s(P_{\cX}R_{Y})+s(P_{\cY}R_{X})}{\cos(\Theta(\cX,\cY))}
\prec_w  \frac{2\,\spr(A,\cX+\cY)\, \sin(\Theta(\cX,\cY))}{\cos(\Theta(\cX\coma \cY))} \,.$$
The result follows from combining these last two inequalities.
\end{proof}

\pausa
The next result illustrates the quadratic dependance of $s(P_{\cX}R_{Y})$ from $\sin(\Theta(\cX,\cY))$ in case $\cX$ is $A$-invariant.

\begin{proposition}\label{proSin}
Let $A\in\H(d)$,$\cX,\cY\subset \C^d$ subspaces with $\dim(\cX)=\dim(\cY)=k$. Assume that  $\cX$ is $A$-invariant.
Then, 
\begin{equation}\label{eqn21}
s(P_{\cX}R_{Y})\prec_w 2\ (\la_i(A_{\cX+\cY})- \lambda_{\min}(A_{\cX+\cY}))_{i\in\I_k} \  \sin^2(\Theta(\cX,\cY))\,.
\end{equation}
\end{proposition}

\begin{proof}
 See the Appendix (Section \ref{sec append}).
\end{proof}

 \begin{theorem}\label{theorem bound inv nostro}
 Let $A\in \H(d)$, $\cX,\cY\subset\C^d$ subspaces, $\dim(\cX)=\dim(\cY)=k$, and assume that $\cX$ is $A$-invariant. If $\Theta_1(\cX,\cY)<\frac{\pi}{2}$, then 
\begin{equation}\label{invariantsen1} |\lambda(\rho(X)) - \lambda(\rho(Y))|\prec_w 
\frac{2\  (\la_i(A_{\cX+\cY})- \lambda_{\min}(A_{\cX+\cY}))_{i\in\I_k}\  \sin^2(\Theta(\cX,\cY))  }{\cos(\Theta(\cX\coma \cY))} \,.
\end{equation}
 \end{theorem}
 
 \begin{proof}
 The result follows from Corollary \ref{cor caso invariante} and Proposition \ref{proSin} with an argument similar to that in the proof of Theorem \ref{theorem apriori nostro} above.
 \end{proof}

\begin{corollary}\label{cor bounds nostro}
 Let $A\in \H(d)$, $\cX,\cY\subset\C^d$ subspaces, $\dim(\cX)=\dim(\cY)=k$. If $\Theta_1(\cX,\cY)<\frac{\pi}{2}$, then 
$$
|\lambda(\rho(X)) - \lambda(\rho(Y))|\prec_w \frac{2}{\cos(\Theta_1(\cX,\cY))}\ 
{\rm Spr}(A,\cX+\cY)\,\sin(\Theta(\cX,\cY)) \,.
$$
If we assume further that $\cX$ is $A$-invariant, then
$$|\lambda(\rho(X)) - \lambda(\rho(Y))|\prec_w 
\frac{2}{\cos(\Theta_1(\cX,\cY))} \  (\la_i(A_{\cX+\cY})- \lambda_{\min}(A_{\cX+\cY}))_{i\in\I_k} \ 
\sin^2(\Theta(\cX,\cY)) \,. $$
\qed
\end{corollary}

\pausa
We end this section with some remarks concerning the relations among 
Theorems \ref{theorem apriori nostro}
 and \ref{theorem bound inv nostro}, Corollary \ref{cor bounds nostro}
and the conjectured bounds in \eqref{eq conj apr1} and \eqref{eq conj apr2}. As already mentioned in 
Remark \ref{rem prepa1}, the bounds in \eqref{eq conj apr1} and \eqref{eq conj apr2}
would be particularly relevant in case $\cY$ is a (small) perturbation of $\cX$ or, in other terms, in case that 
$\cX$ and $\cY$ are close subspaces (e.g. $\Theta_1(\cX,\cY)$ is small). In order to simplify the discussion, let us assume that 
$\Theta_1(\cX,\cY)\leq \pi/4$. We point out that this assumption holds in a number of significant situations 
(see for example \cite[Section 5.2.]{ZK}). In this case, if $A\in\matsad$ then Corollary \ref{cor bounds nostro}
implies that 
\beq\label{eq cons cor conj1}
|\lambda(\rho(X)) - \lambda(\rho(Y))|\prec_w \,(2\,\sqrt 2)\ {\rm Spr}(A,\cX+\cY)\, \sin(\Theta(\cX,\cY)) \,.
\eeq  Hence, under the present assumptions ($\Theta_1(\cX,\cY)\leq \pi/4$), 
the upper bound in \eqref{eq cons cor conj1} has the conjectured order of approximation (when considering
$A$ as well as the subspaces $\cX$ and $\cY$ as variables), up to the constant factor $2\,\sqrt 2$. 

\pausa
If we further assume that $\cX$ is $A$-invariant then by the same result we get that 
\beq\label{eq cons cor conj2}
|\lambda(\rho(X)) - \lambda(\rho(Y))|\prec_w 
\,(2\,\sqrt 2) \  (\la_i(A_{\cX+\cY})- \lambda_{\min}(A_{\cX+\cY}))_{i\in\I_k} \  
\sin^2(\Theta(\cX,\cY)) \,. 
\eeq Again, the upper bound in \eqref{eq cons cor conj2} has the conjectured order of approximation (when considering
$A$ as well as the subspaces $\cX$ and $\cY$ as variables), up to the constant factor $2\,\sqrt 2$. Moreover, notice that
this bound holds for an arbitrary $A$-invariant subspace $\cX$ (as opposed the bound in \eqref{eq theorem apr3} from \cite{AKFEM} that is
shown to hold for special choices of $A$-invariant subspaces $\cX$).

\subsection{The tan$\,\Theta$ theorem revisited: improved quadratic a posteriori error bounds}\label{sec 3.3.}

In this section we revisit Nakatsukasa's extension of Davis-Kahan's $\tan(\theta)$ theorem. Our motivation is
the study of an improved version of this result conjectured in \cite{ZK} 
(see Corollary \ref{coro conjKZ tan} below). 
We first recall the separation hypothesis 
for Nakatsukasa's result. As before, in this section we adopt Setting \ref{nota1}.

\begin{definition}\label{sepa DK}\rm 
Let $A\in\matsad$ and let 
$\mathcal X,\,\mathcal Y\subset \C^d$ be subspaces
with $\dim \mathcal X=\dim \mathcal Y=k$, such that $\mathcal X$ is $A$-invariant. Let
$[X,X_\perp],\, [Y,Y_\perp]\in\matud$ be unitary matrices such that the columns of (the $d\times k$ matrices)
$X$ and $Y$ form ONB's of $\mathcal X$ and $\mathcal Y$ respectively. Given $\delta>0$ we say that $(A\coma \mathcal X\coma \mathcal Y\coma \delta)$
satisfies the Davis-Kahan-Nakatsukasa (DKN) separation property if there exist $a\leq b$ such that 
\begin{enumerate}
\item $ \lambda_i(X_\perp^*AX_\perp)=\lambda_i(P_{\cX^\perp}\, A\, P_{\cX^\perp}) \in[a,b]$, for $i\in\I_{d-k}$;
\item $\lambda_i(Y^*AY)=  \lambda_i(P_{\cY}\, A\, P_{\cY})  \in (\infty, a-\delta]\cup [b+\delta,\infty)$, for $i\in\I_{k}$.
\EOE
\end{enumerate}
\end{definition}
\pausa

\pausa
Next we state Nakatsukasa's $\tan\Theta$ theorem under relaxed conditions.

\begin{theorem}[\cite{Nakats}]\label{theorem tan Nakats}
Let $A\in\matsad$, \ $\mathcal X,\,\mathcal Y\subset \C^d$ and let 
$\delta>0$ be such that $(A\coma \mathcal X\coma \mathcal Y\coma \delta)$
satisfies the DKN separation property. Then, $\Theta_1(\cX,\cY)<\pi/2$ and 
$$ \delta\, \|\tan(\Theta(\mathcal X\coma \mathcal Y))\|\leq \| R_Y\| \,,$$ 
for every unitarily invariant norm $\|\cdot\|$. Equivalently, $\delta\,\tan(\Theta(\mathcal X\coma \mathcal Y))\prec_w s(R_Y)$. 
\end{theorem}

\begin{rem}\rm 
Theorem \ref{theorem tan Nakats} requires the knowledge of the full matrix $A$ in order to bound the (norm of the) vector
$\tan(\Theta(\cX,\cY))$ from above. Instead, it would be interesting to bound the vector
$\tan(\Theta(\cX,\cY))$ from above (only) in terms of the self-adjoint operator
$A_{\cX+\cY}=P_{\cX+\cY} A |_{\cX+\cY} 
 \in \mathcal L(\cX+\cY)$ (defined in the obvious way). 
In the next result we show that the $\tan\Theta$ theorem
mentioned above allow to obtain such a result. Moreover, we will also see that it is possible to describe separation 
hypothesis for $(A_{\cX+\cY},\,\cX,\, \cY)$, that are more general than
the DKN separation hypothesis for $(A,\,\cX,\, \cY)$, for which the $\tan\Theta$ theorem holds; 
arguing in terms of interlacing inequalities, we can show that these separation hypotheses on $A_{\cX+\cY}$ provide better separation constants
than the DKN separation hypotheses on the matrix $A$.
\EOE
\end{rem}
\pausa
We formalize the content of the previous remark - with a small variation on the notation - in the following result.
First, we recall some facts related with the relative position of two subspaces.
\begin{rem}\label{rem two subspaces}\rm 
Let $\cX,\, \cY\subset \C^d$ be two subspaces with $\dim\cX=\dim\cY=k$. Consider the mutually orthogonal subspaces
$$\cH_{00}=\cX^\perp\cap\cY^\perp \ , \ \cH_{10}=\cX\cap\cY^\perp \ , \ \cH_{01}=\cX^\perp\cap\cY \ , \ 
\cH_{11}=\cX\cap\cY \ , $$and $\cH_g=\C^d\ominus (\cH_{00}\oplus \cH_{10}\oplus \cH_{01}\oplus \cH_{11})$ which is called the {\it generic
part} of the pair $(\cX,\cY)$. Each of these five (possible zero)
subspaces reduces each projection $P_\cX$ and $P_\cY$. Moreover, the subspaces $\cX_g=\cX\cap \cH_g$ and $\cY_g=\cY\cap \cH_g$ 
are in {\it generic position} so that $\H_g=\cX_g+\cY_g$. For details of this well known construction and several fundamental results 
see \cite{Hal}.
\EOE
\end{rem}

\begin{theorem}\label{theorem tantan mejorado1}
Let $A\in\matsad$, and let 
$\mathcal X,\,\mathcal Y\subset \C^d$ be such that $\dim \cX=\dim\cY=k$. Let $A_{\cX+\cY}=S^* A S\in\cH(p)$, 
where $S\in\cM_{d,p}(\C)$ is such that its 
columns form an ONB for $\cX+\cY$. Then,
\begin{enumerate}
\item  If $\delta>0$ is such that $(A\coma \mathcal X\coma \mathcal Y\coma \delta)$
satisfies the DKN separation property then there exists $\delta\,'\geq \delta$ such that 
$(A_{\cX+\cY}\coma S^*\mathcal X\coma S^*\mathcal Y\coma \delta\,')$ satisfies the DKN separation property. 
\item If $\delta\,'>0$ is such that 
$(A_{\cX+\cY}\coma S^*\mathcal X\coma S^*\mathcal Y\coma \delta')$ satisfies the DKN separation property, then
\beq\label{eq theorem tan tan mejorado}
 \delta\,'\, \|\tan(\Theta(\mathcal X\coma \mathcal Y))\|\leq \| A_{\cX+\cY}\,Y_S-Y_S\,(Y_S^*A_{\cX+\cY}\,Y_S)\| 
=\|P_{\mathcal X+\mathcal Y} \ R_Y\| \eeq
for every unitarily invariant norm $\|\cdot\|$, where $Y_S=S^*Y\in\cM_{p,k}(\C)$.
\end{enumerate}
\end{theorem}
\proof We first show item 1 Let $X,\, Y\in \mathcal M_{d,k}(\C)$ be such that their columns form orthonormal 
bases of $\cX$ and $\cY$, respectively. By hypothesis, there exist $a\leq b$ such that: for $i\in\I_{d-k}$ and 
$j\in\I_{k}$ we have that
$$\lambda_i(X_\perp^*AX_\perp)\in[a,b] \py 
\lambda_j(Y^*AY) \in (\infty, a-\delta]\cup [b+\delta,\infty)\,,$$
where $X_\perp\in\cM_{d,d-k}(\C)$ is such that its columns for an ONB for $\cX^\perp$.
Let $\cZ=\cX+\cY$ and notice that $S\in\cM_{d,p}(\C)$ is an isometry from $\C^p$ onto $\cZ$. Moreover, the matrix
$S^*A\,S\in\cH(p)$. Similarly, $X_S=S^*X,\, Y_S=S^*Y\in \cM_{p,k}$ 
are isometries from $\C^k$ onto 
$S^*\cX,\, S^*\cY\subseteq \C^p$, respectively. 
Consider the mutually orthogonal subspaces 
$$\cH_{11}=\cX\cap \cY \peso{,} \cX_g= \cH_g\cap \cX \py \cX_{g^\perp}= \cH_g\ominus \cX_g \,,$$
where $\cH_g$ is the subspace of $\C^d$ corresponding to the generic part of the pair $(\cX\coma \cY)$ (see Remark \ref{rem two subspaces}).
By Theorem \ref{theorem tan Nakats} we have that $\Theta_1(\cX,\cY)<\pi/2$ so then, $\cX^\perp\cap \cY=\{0\}=\cX\cap\cY^\perp$.
Thus, $$\cX=\cH_{11}\oplus \cX_g \peso{,} \cZ=\cH_{11}\oplus \cX_g\oplus \cX_{g^\perp}\py \cX_{g^\perp}=\cZ\ominus \cX\,.$$ 
Let $X'\in M_{d,(p-k)}(\C)$ be such that its columns form an orthonormal 
basis of $\cX_{g^\perp}\subset \cX^\perp$. Then, $X'_S=S^*\,X'\in\cM_{p,(p-k)}(\C)$ is an isometry 
from $\C^{p-k}$ onto $S^*\cX_{g^\perp}=(S^*\cX)^\perp\subseteq \C^p$. 
To check the DKN separation property for
 $(A_{\cX+\cY}\coma S^*\mathcal X\coma S^*\mathcal Y)$ we consider the eigenvalues of 
$$ (X'_S)^* (S^*A\,S)\, X'_S= (X')^* \, S\,S^*\, A\, S\, S^*\, X'= (X')^* \,  A\,  X'\in \cH(p-k)\, ,$$ since
$SS^*=P_\cZ\in \mat$, $P_\cZ \, X'=X'$ and $ (X')^*\, P_\cZ =(X')^*$. Hence, we now see that 
$$ \la_i((X'_S)^* (S^*A\,S)\, X'_S) = \la_i(P_{\cX_{g^\perp}} A\,  P_{\cX_{g^\perp}})\peso{for} i\in \I_{p-k}\,.$$
Since $\cX_{g^\perp}\subset \cX^\perp$ we have that $P_{\cX_{g^\perp}} A\,  P_{\cX_{g^\perp}}$ 
is a compression of $P_{\cX^\perp} A\, P_{\cX^\perp}$. Using the interlacing inequalities
for compressions of self-adjoint matrices (see \cite{bhatia}), we get that 
if $\la_i((P_{\cX^\perp} A\, P_{\cX^\perp}))\in [a,b]$, for $i\in \I_{d-k}$, then
\beq \label{eq theorem tan pulenta1}
 \la_i(P_{\cX_{g^\perp}} A\,  P_{\cX_{g^\perp}})\in [a,b] \peso{for} i\in \I_{p-k}\,.
\eeq On the other hand, notice that 
$$ Y_S^* \,(S^*A\, S) \, Y_S= Y^* P_\cZ A\, P_\cZ\, Y= Y^*A\, Y$$
since, as before, $SS^*=P_\cZ$, $P_\cZ Y=Y$ and $Y^* P_\cZ= Y^*$. Therefore, we get that 
\beq \label{eq theorem tan pulenta2}
 \la_i(Y_S^* \,(S^*A\, S) \, Y_S)=\la_i(Y^*A\,Y)\in (\infty, a-\delta]\cup [b+\delta,\infty) 
\peso{for} i\in\I_{k}\, .
\eeq 
Item 1 now follows from \eqref{eq theorem tan pulenta1} and \eqref{eq theorem tan pulenta2} and
the fact that $S^*\cX\subseteq \C^p$ is, by construction, 
an $A_{\cX+\cY}$-invariant subspace.

\pausa In order to show item 2, we fix a unitarily invariant norm $\|\cdot\|$. Using that $\cX,\,\cY\subset \cZ$ and the fact that 
$S^*$ is an isometry from $\cZ$ onto $\C^p$, we see that $\Theta(\cX,\cY)=\Theta(S^*\cX,S^*\cY)$.
Then, an application of Nakatsukasa's $\tan\Theta$ theorem (Theorem \ref{theorem tan Nakats}) 
to the self-adjoint matrix $S^*AS\in\cH(p)$ and subspaces $S^*\cX,\, S^*\cY\subseteq \C^p$ shows that 
$$ \delta\,'\, \|\tan(\Theta(\mathcal X\coma \mathcal Y))\|\leq \| A_{\cX+\cY}\,Y_S-Y_S\,(Y_S^*A_{\cX+\cY}\,Y_S)\,\| \,,$$
where $Y_S=S^*Y\in\cM_{p,k}$ is an isometry from $\C^k$ onto $S^*\cY$. We notice that 
\begin{eqnarray*}
A_{\cX+\cY}\,Y_S-Y_S\,(Y_S^*A_{\cX+\cY}\,Y_S) &=& S^* A\,S \, S^* Y - S^* Y \,(Y^*S (S^* A\, S) S^*Y)\\
&=& S^* \,(A \,Y - Y\, (Y^* A \,Y))\, ,
\end{eqnarray*} where we have used that $SS^*=P_\cZ$, $P_\cZ\, Y= Y$ and $Y^*\, P_\cZ=Y^*$. 
Hence, it follows that 
$$
\| A_{\cX+\cY}\,Y_S-Y_S\,(Y_S^*A_{\cX+\cY}\,Y_S)\|=\| P_\cZ  \,(A\, Y - Y\, (Y^* A \,Y))\|
=\| P_{\cX+\cY}\, R_Y\|\ . 
$$
\QED

\begin{rem}\rm 
With the notation of Theorem \ref{theorem tantan mejorado1} and using Remark \ref{Domkyfan}, 
\eqref{eq theorem tan tan mejorado} is equivalent 
to the majorization relation  
$$\delta\,'\, \tan(\Theta(\mathcal X\coma \mathcal Y)\prec_w s( A_{\cX+\cY}\,Y_S-Y_S\,(Y_S^*A_{\cX+\cY}\,Y_S))
=s(P_{\mathcal X+\mathcal Y} \ R_Y)\,  $$ in terms of the separation constant $\delta'$ for $A_{\cX+\cY}=S^*A\,S$, 
$S^*\cX$ and $S^*\cY$. \EOE
\end{rem}

\pausa
Consider the notation in Theorem \ref{theorem tantan mejorado1}. Let $\delta>0$ be 
such that $(A\coma \mathcal X\coma \mathcal Y\coma \delta)$ satisfies the DKN separation property. Given a unitarily invariant norm $\|\cdot\|$, Theorem \ref{theorem tan Nakats} allows to bound $\|\tan \Theta(\cX,\cY)\|$ from above by
\beq \label{eq rem boun tan1}
\|\tan \Theta(\cX,\cY)\|\leq \frac{\|R_Y\|}{\delta}\,.
\eeq On the other hand, by item 2 in Theorem \ref{theorem tantan mejorado1}
there exists $\delta'\geq \delta>0$ such 
that $(A_{\cX+\cY}\coma S^*\mathcal X\coma S^*\mathcal Y\coma \delta')$ satisfies the DKN separation property, so that 
we get the upper bound 
\beq \label{eq rem boun tan2}
\|\tan \Theta(\cX,\cY)\|\leq \frac{\|P_{\cX+\cY}\, R_Y\|}{\delta\,'}\,.
\eeq Since $\|P_{\cX+\cY}\, R_Y\|\leq \|R_Y\|$ and $\delta\leq \delta\,'$, we immediately see that 
the upper bound in \eqref{eq rem boun tan2} improves the classical bound in \eqref{eq rem boun tan1}.
In order to compare these two bounds in some more detail, let us consider the following 
\begin{exa}\label{exa2}\rm 
Let $\tilde \lambda=(a,b,d,c)\in\R^4$, where $a<b<c<d$, and let $\tilde A\in\cH(4)$ be given by
$\tilde A=D_{\tilde \lambda}$. For the purposes of this 
example, we consider the real parameter $c\in (b,d)$ as variable (while $a,\,b,\,d$ are fixed).

\pausa
Let $\cX,\,\cY_\theta\subset \C^4$ be as in Example \ref{exa1} i.e. 
$\cX=\text{span}\{e_1,\,e_2\}$ and $\cY_\theta=\text{span}\{e_1,\,f_\theta\}$. Recall that 
$\Theta(\cX,\cY_{\theta})=(\theta,0)$. In particular, 
$\tan \Theta(\cX,\cY_{\theta})=(\tan \theta,0)$ in this case.

\pausa
It is clear that 
$\cX+\cY_{\theta}=\text{span}\{e_1,\,e_2,\, e_3\}$.
Let 
$$
X=\begin{pmatrix}
1 & 0 \\
0 & 1 \\
0 & 0 \\
0 & 0 
\end{pmatrix} 
\peso{,}
X_{\perp}=\begin{pmatrix}
0 & 0 \\
0 & 0 \\
1 & 0 \\
0 & 1 
\end{pmatrix} 
\py 
Y_\theta=\begin{pmatrix}
1 & 0 \\
0 & \cos \theta \\
0 & \sin \theta \\
0 & 0 
\end{pmatrix} \,.
$$ 
Then, we have that $\lambda(Y_\theta^* \tilde A Y_\theta)=(b\,\cos^2\theta + d\, \sin^2(\theta) , a)$, while 
$\lambda(X_\perp ^* \tilde A\, X_\perp)=(d,c)$. Therefore, if we 
let $\theta_0(c)=\theta_0=\arcsin\left(\sqrt{\frac{c-b}{d-b}}\right)$ and set 
$$\delta_\theta= c-(b\,\cos ^2 \theta + d\, \sin^2\theta)>0 \peso{for} 0<\theta<\theta_0\, , $$ then 
$(\tilde A,\cX,\cY_\theta,\delta_\theta)$ satisfies the DKN separation property, and $\delta_\theta$ is the optimal 
(largest) separation constant and the separation property holds only for $0<\theta<\theta_0$ in this case.  Again, simple computations show that 
$s(R_{Y_\theta})=( (d-b) \cos \theta \, \sin\theta, 0)$.

\pausa
Now, \eqref{eq rem boun tan1} obtained from Theorem \ref{theorem tan Nakats} becomes
\beq \label{eq conc tan 1}
\tan \theta\leq \frac{(d-b) \cos \theta \, \sin\theta}{c-(b\,\cos ^2 \theta + d\, \sin^2\theta)} \peso{for}
0<\theta<\theta_0\,.
\eeq Notice that $\lim_{c\rightarrow b^+} \theta_0=0$ i.e., the range of $\theta$ for which we can apply the bound in 
\eqref{eq conc tan 1} tend to become small. In the limit case in which $b=c$ (i.e. multiple eigenvalues)
we can not apply the bound \eqref{eq conc tan 1} (the separation constant in this case is $\delta_0=0$). Finally, if we consider the limit case in which $\theta$ becomes small, then
the upper bound is comparable with the upper bound $(\frac{d-b}{c-b})\,\tan \theta \ (>\tan \theta)$.

\pausa
On the other hand, $\cX+\cY_{\theta}\ominus \cX=\C\,e_3$, the subspace spanned by $e_3$. In this case, if we let 
$X'=(0,0,1,0)^t$, it is clear that $\la((X'_S)^* \tilde A \, X'_S)=d$. Therefore, if we let 
$\delta'_\theta= d- (b\,\cos ^2 \theta + d\, \sin^2\theta)= (d-b)\,\cos^2\theta>0$, for $\theta\in (0,\pi/2)$, we get that 
$(\tilde A_{\cX+\cY_{\theta}},S^*\cX,S^*\cY_{\theta},\delta'_\theta)$ 
satisfies the DKN separation property, where $S\in\cM_{4,3}(\C)$ is the matrix whose columns are the first
three elements in the canonical basis. In this case
we have that $$\frac{s_1(P_{\cX+\cY_{\theta}} \, R_{Y_{\theta}})}{\delta'_\theta}= \frac{(d-b) \cos \theta \, \sin\theta}{(d-b)\,\cos^2\theta}=\tan \theta\,,$$
and hence, the upper bound in \eqref{eq rem boun tan2} coincides with $\tan \theta$ (where $\tan \Theta(\cX,\cY_{\theta})=(\tan \theta, 0)$) i.e. the upper bound is sharp. Notice that the bound is applicable for every
$\theta\in (0,\pi/2)$.
\EOE
\end{exa}

\pausa
The following result was conjectured in \cite{ZK}.

\begin{corollary}\label{coro conjKZ tan}
	Let $A\in\matsad$, 
$\mathcal X,\,\mathcal Y\subset \C^d$ and $\delta>0$ be such that $(A\,, \mathcal X\,,\mathcal Y\,,\delta)$
satisfies the DKN separation property. Then, 
$$ \delta\, \|\tan(\Theta(\mathcal X\coma \mathcal Y))\|\leq \|P_{\mathcal X+\mathcal Y} \ R_Y\| \,.$$ 
for every unitarily invariant norm $\| \cdot \|$.
\end{corollary}
\begin{proof}Let 
$S\in\cM_{d,p}(\C)$ be such that its columns form an ONB for $\cX+\cY$. 
By item 1 in Theorem \ref{theorem tantan mejorado1}, there exists $\delta\,'\geq \delta$ such that 
$(S^* A\, S \coma S^* \mathcal X\coma S^* \mathcal Y\coma \delta\,')$ satisfies the DKN separation property. 
By item 2 of the same result, we have that 
$$ \delta\, \|\tan(\Theta(\mathcal X\coma \mathcal Y))\|\leq \delta\,'\, \|\tan(\Theta(\mathcal X\coma \mathcal Y))\|
\leq \|P_{\mathcal X+\mathcal Y} \ R_Y\| \,.$$ 
\end{proof}

\pausa
Finally, we get the following quadratic a posteriori error bound for the simultaneous 
approximation of eigenvalues of $A$ by the Ritz values corresponding to Rayleigh quotients for 
which a DKN separation property holds.
\begin{theorem}\label{theorem aplic tantan}
Let $A\in\matsad$, 
$\mathcal X,\,\mathcal Y\subset \C^d$ and $\delta>0$ be such that $(A\coma \mathcal X\coma \mathcal Y\coma \delta)$
satisfies the DKN separation property. Then, for every unitarily invariant norm $\|\cdot\|$ we have that 
$$ 
\|\la (\rho(X))-\la(\rho(Y))\|\leq \frac{\| P_{\cX+\cY} \, R_Y\|^2}{\delta}\,.
$$ 
\end{theorem}
\begin{proof}
This is a consequence of Corollary \ref{cor caso invariante} and Theorem \ref{theorem tantan mejorado1}.
\end{proof}

\pausa
Theorem \ref{theorem aplic tantan} allows to obtain the following extension of
\cite[Theorem 5.3]{ZAK} (see Remark \ref{rem aplic conj tan2} below) which is a quadratic a posteriori majorization error bound for simultaneous approximation of consecutive eigenvalues.

\begin{corollary}\label{cor ZAK extend}
Let $A\in\matsad$ and let $\mathcal Y\subset \C^d$ be such that:
\begin{enumerate} 
\item $\la_1(Y^*AY)<\la_j(A)$, where $j\in\I_{d-k}$ is the smallest such index;
\item $\la_i(Y^*AY)\geq \la_{i+j}(A)$, for $i\in\I_k$.
\end{enumerate}
 Let $\cU$ be the $A$-invariant space spanned by 
the eigenvectors associated with $\la_i(A)$, for $1\leq i\leq j$, and set $\cX=(I-P_U)\cY$. 
If  $\eta=\la_j(A)-\la_1(Y^*AY)>0$ then  
$$ 
\| (\la_{i+j}(A))_{i\in\I_k}-\la(\rho(Y))\|\leq \frac{\| P_{\cX+\cY} \, R_Y\|^2}{\eta}\,,
$$ for every unitarily invariant norm $\|\cdot\|$.
\end{corollary}
\begin{proof}
Let $\cV=\cU+\cY$ and notice that $\cU\cap\cY=\{0\}$; hence, $p=\dim\cV=\dim \cU+k$ i.e. $j=\dim\cU=p-k$.
 Moreover, $\cV\ominus \cU=(I-P_\cU)\cY=\cX$; then, in particular, 
$\dim\cX=\dim\cY$ and $\cV\ominus \cX=\cU$. Also notice that $\Theta_1(\cX,\cY)<\pi/2$ or otherwise, we would have that 
$\cU\cap \cY\neq \{0\}$, since $\cV\ominus \cX=\cU$.

\pausa
Let $V\in \cM_{d,p}(\C)$ be such that its columns form an ONB of $\cV$ 
and set $A_V=V^*AV\in\cH(p)$. Similarly, let $X,\, Y\in\cM_{d,k}(\C), \, U\in \cM_{d,p-k}(\C)$
 be such that their columns form ONB's of $\cX$, $\cY$ and $\cU$ respectively; set $X_V=V^*X,\, Y_V=V^*Y\in\cM_{p,k}(\C)$ and $U_V=V^*U\in \cM_{p,p-k}(\C)$. 
Then, the columns of $U_V$ span $\cU_V\subset \C^p$ an $A$-invariant space of $A_V$. In particular, the columns of $X_V$ span 
$\cX_V\subset \C^p$ which is also an $A$-invariant space of $A_V$. In this case $\cX_V^\perp=\cU_V$ and
 $\Theta_1(\cX_V,\cY_V)=\Theta_1(\cX,\cY)<\pi/2$, where $\cY_V\subset\C^p$ 
is the space spanned by the columns of $Y_V$. 
Notice that, by construction $\la_i(Y_V^*A_V\, Y_V)=\la_i(Y^*A\, Y)$, for $i\in\I_k$.
Since $\cX\subset \cU^\perp$ by the interlacing inequalities for 
compressions of self-adjoint matrices and item 2 above, we get that   
for $i\in\I_k$,
\beq \label{eq theorem ultimo momento1}
 \la_i(X_V^*A_V\, X_V)=	\la_i(X^*A\, X)\leq \la_i(A_{U_\perp})=\la_{j+i}(A)\leq  \la_i(Y_V^*A_V\, Y_V)\, ,
\eeq
where $U_\perp\in \cM_{d,d-j}(\C)$ is such that its columns for an ONB for $\cU^\perp$.
On the other hand, by hypothesis $(A_V, \cX_V,\cY_V,\eta)$ satisfies the DKN separation property 
(recall that $\cX_V^\perp=\cU_V$). Hence, by Theorem \ref{theorem aplic tantan}
we conclude that 
\beq \label{eq theorem ultimo momento2}
\| \la(X_V^*A_V\, X_V)-\la(Y_V^*A_V\, Y_V)\| \leq \frac{\| \, P_{\cX_V+\cY_V} 	\, (A_V\, Y_V - Y_V\, (Y_V^*A_V\, Y_V)) \,\|^2  }{\eta}\,.
\eeq
By \eqref{eq theorem ultimo momento1} we get that 
$$| (\la_{i+j}(A))_{i\in\I_k}-\la(Y_V^*A_V\, Y_V)  |\prec_w | \la(X_V^*A_V\, X_V)-\la(Y_V^*A_V\, Y_V)| \,.$$
On the other hand, arguing as in the proof of Theorem \ref{theorem tantan mejorado1} we see that 
$$ \| \, P_{\cX_V+\cY_V} 	\, (A_V\, Y_V - Y_V\, (Y_V^*A_V\, Y_V)) \,\|= \| \, P_{\cX+\cY} 	\, R_Y\,\|\,.$$
The result follows from these last facts together with \eqref{eq theorem ultimo momento2} and Remark \ref{Domkyfan}.
\end{proof} 

\begin{rem}\label{rem aplic conj tan2}\rm 
We mention that the hypothesis in item 1 in Corollary \ref{cor ZAK extend} is that 
there exists an eigenvalue $\beta$ of $A$ such that $\la_1(Y^*AY)<\beta$.
Indeed, in this case we can apply the interlacing inequalities and get that 
$\la_i(Y^*AY)\geq \la_{d-k+i}(A)$, for $i\in\I_k$. Therefore, $\beta=\la_j(A)$ for some 
$1\leq j\leq d-k$. 

\pausa
The hypothesis in item 2 is rather restrictive and difficult to check in general. 
 Nevertheless, we mention two cases in which the hypotheses in Corollary \ref{cor ZAK extend} can be easily checked:
\begin{enumerate}
\item In case the hypothesis in item 1 holds for $j=d-k$, by the interlacing inequalities we have 
$$  \la_i(Y^*AY)\geq \la_{i+d-k}(A) \peso{for} i\in\I_k \, ,$$ so the hypothesis in item
2 automatically hold.
\item
In case $k=1$ that is, if $\cY=\C\,y$ for a unit norm
vector $y\in\C^d$, the hypotheses become the existence of $j\in\I_{d-1}$ such that 
$ \la_{j+1}(A)\leq \langle A\, y,\, y\rangle<\la_j(A)$; then, Corollary \ref{cor ZAK extend} implies that 
$$  0\leq \langle A y,\, y\rangle - \la_{j+1}(A)\leq \frac{\| P_{\cX+\cY}( Ay-\langle A y,\, y\rangle\, y)\|}{\la_{j}(A)-\langle A y,\, y\rangle}\, ,$$
where $\cX=\C\,x$, for $x=(I-P_U)y\in\C^d$; this is \cite[Theorem 5.3]{ZAK}. As explained in \cite{ZAK}, Corollary  
\ref{cor ZAK extend} encodes several known bounds related with eigenvalue estimation even when $k=1$. 
\EOE
\end{enumerate}
\end{rem}
\section{Appendix}\label{sec append}

Here we collect several and well known results about majorization, used throughout our work.
The first result deals with submajorization relations between singular values of arbitrary matrices in $\mat$.
For detailed proofs of these results and general references in majorization theory see \cite{bhatia,HJ,MaOlAr}.
For $A\in \mat$ we denote by $\text{re}(A) = \frac{A+A^*}{2}\in \matsad$. 

\begin{theorem}\label{theorem ag}\rm Let $C,\,D\in \mat$. Then,
\begin{enumerate}
\item $s(C+D)\prec_w s(C)+s(D)$;  \hfill(Lidskii's additive property)
\item $s(\text{re}(C))\prec_w s(C)$;
\item $s(CD)\prec_w s(C)\, s(D)$; \hfill(Lidskii's multiplicative property)
\item If we assume that $CD\in\matsad$ then $s(CD)\prec_w s(\text{re}(DC))$.
\end{enumerate} \qed
\end{theorem}

\pausa
For hermitian matrices we have the following majorization relations

\begin{theorem}\label{theorem ah} \rm Let $C,\, D\in \matsad$. Then,
\begin{enumerate}
\item $\lambda(C)-\lambda(D)\prec \lambda(C-D)\prec\lambda(C)-\lambda^{\uparrow}(D)$;
\item $|\la(C)-\la(D)|\prec_w s(C-D)$;
\item Let $\cP=\{P_j\}_{j=1}^r$ be a system of projections (i.e. they are mutually orthogonal projections on $\C^d$ 
such that $\sum_{i=1}^r P_i=I$). If $C_{\cP}(C)=\sum_{i=1}^r P_iC P_i$, then $\lambda(C_{\cP}(C))\prec \lambda(C)$.
\end{enumerate} \qed
\end{theorem}
\pausa
In the next result we describe elementary but useful properties of (sub)majorization between real vectors.

\begin{lemma}\label{lemma submaj props1}\rm Let $x,\, y,\,z\in \R^k$. Then,
\begin{enumerate}
\item $x\da + y\ua\prec x+y\prec x\da+y\da$;
\item If $x\prec_w y$ and $y,\, z\in(\R^k)\da$ then $x+z\prec_w y+z$;
\end{enumerate}
If we assume further that $x,\,y,\, z\in \R_{\geq 0}^k$ then,
\begin{enumerate}
\item[3.] $x\da\, y\ua\prec_w x\, y\prec_w x\da\, y\da$;
\item[4.] If $x\prec_w y$ and $y,\, z\in (\R_{\geq 0}^k)\da$ then $x\, z\prec_w y\, z$.\qed
\end{enumerate}
\end{lemma}

\begin{proposition}\label{hat trick como en el futbol}\rm
Let $1\leq k<d$ and let $E\in \M_{k,(d-k)}(\C)$. Then $$ \hat{E}=\begin{pmatrix}
0&E\\ E^*&0
\end{pmatrix}\in \H(d) \py \la(\hat E)= (s(E),-s(E^*)\da)\in(\R^{d})\da\,.$$\qed
\end{proposition}

\begin{theorem}[\cite{AKFEM}]\label{theoremr Fem4.6}\rm 
Let $\cX,\, \cY\subset \C^d$ be such that $\dim(\cX)=\dim(\cY)=k$. Then 
$$
\la(P_{\cX}P_{\cY^{\perp}}P_{\cX})= s(P_{\cX}P_{\cY^{\perp}}P_{\cX}) = 
s^2(P_{\cY}P_{\cX^{\perp}})=s^2(P_{\cX^{\perp}}P_{\cY})=(\sin^2(\Theta(\cX,\cY)),0_{d-k}).$$\qed
\end{theorem}
\pausa
Notice that item 2 below is Theorem \ref{lemmaa3'} from Section \ref{sec maj err}.

\begin {theorem}\label{theorem mas pulenta que 31}
Let $C,\,D\in \cH(k)$. Then,
\begin{enumerate}
\item if $T\in \cG l(k)^+$, then 
$s(C-D)\prec_w s(T^{-1})\, s(CT-TD)\,.$
\item if $T\in \cG l(k)$, then 
$|\la(C)-\la(D)|\prec_w s(T^{-1})\, s(CT-TD)$.
\end{enumerate}
\end{theorem}
\begin{proof}
We first show item 1 Since $T$ is positive and invertible, 
using Theorem \ref{theorem ah} (item 3) we get that 
 \begin{eqnarray*}
 s(C-D)&=&s(CT^{\frac{1}{2}}T^{-\frac{1}{2}}-T^{-\frac{1}{2}}T^{\frac{1}{2}}D))=
s(T^{-\frac{1}{2}}(T^{\frac{1}{2}}CT^{\frac{1}{2}}-T^{\frac{1}{2}}DT^{\frac{1}{2}})T^{-\frac{1}{2}})\\ 
&\prec_w& s(T^{-\frac{1}{2}})^2 \, s(T^{\frac{1}{2}}CT^{\frac{1}{2}}-T^{\frac{1}{2}}DT^{\frac{1}{2}}) =
s(T^{-1}) \, s(T^{\frac{1}{2}}( C- D)\,T^{\frac{1}{2}})\,.
 \end{eqnarray*}
By Theorem \ref{theorem ag} (items 2 and 4) and the fact that 
$\text{re}(DT)=\text{re}(TD)$ we obtain that 
 \begin{equation}\label{eqn6 majorization}
 s(T^{\frac{1}{2}}(C-D)T^{\frac{1}{2}})\prec_w s(\text{re}[(C-D)T])=s(\text{re}[CT-T D])\prec_w s(CT-TD),
 \end{equation}
 By the previous inequalities and Lemma \ref{lemma submaj props1} we see that
\begin{equation}\label{proof lemma3'}
  s(C-D)\prec_w s(T^{-1})\, s(CT-T D)\,.
  \end{equation}
\pausa In order to show item 2, consider a representation of $T$ given by 
$T=U\Sigma V^*$, where $U,\, V\in \cU(k)$ are unitary matrices and $\Sigma\in\cM_k(\C)$ is the diagonal
matrix with  main diagonal $s(T)\in \R_{\geq 0}^k$ (notice that such representation follows from the SVD decomposition of $T$); note that $\Sigma$ is 
definite positive and invertible.
Using item 2 in Theorem \ref{theorem ah} and (the already proved) item 1 of the statement we get
 \begin{align}\label{eqn8}
|\la(C)-\la(D)|&= |\la(U^*CU)-\la(V^*DV)|\prec_w s(U^*CU-V^*DV)\nonumber 
\\&\prec_w s(\Sigma^{-1})\, s(U^*CU\Sigma-\Sigma V^*DV)=s(T^{-1})\, s(U^*(CT-TD)V)\nonumber \\ \nonumber &=s(T^{-1})\, s(CT-TD)\,.
 \end{align}
\end{proof}

\pausa
In what follows we re-state and prove two propositions of Section \ref{subsec applic}.

\medskip
\noi
{\scshape Proposition \ref{prospread}.}  
Let $A\in \H(d)$ and let $\cX,\cY\subset \C^d$ with $\dim (\cX)=\dim(\cY)=k$. Then
\begin{equation}\label{eqn12 bis}
s(P_{\cX}\, R_Y)\prec_w {\rm Spr}(A,\cX+\cY) \, \sin(\Theta(\cX,\cY))\,.
\end{equation}

\begin{proof} 
 We begin with a simple reduction argument. In order to describe this reduction it will be convenient to consider
matrices in terms of the linear operators that they induce. Hence, given $A\in\mat$, we consider $A\in\mathcal L(\C^d)$
(defined in the obvious way). The advantage in considering $A\in\mathcal L(\C^d)$ is that we can get different block matrix 
representations of $A$ (considered as an operator) with respect to orthogonal decompositions $\C^d=\cV\oplus\cV^\perp$ for 
a (proper) subspace $\cV\subset \C^d$, in the usual manner. We now proceed as follows:
Let $\cZ=\cX+\cY$ with $\dim \cZ=p$, and consider the matrix representations with respect to the decomposition $\C^d=\cZ\oplus \cZ^\perp$:
$$P_\cX=\begin{pmatrix} P^\cX & 0 \\ 0 & 0\end{pmatrix} \ \coma \
P_\cY=\begin{pmatrix} P^\cY & 0 \\ 0 & 0\end{pmatrix} \py A=\begin{pmatrix} A_{\cZ} & * \\ * & *\end{pmatrix}\,,
$$ where $P^\cX,\, P^\cY,\, A_\cZ=P_{\cZ} A|_\cZ\in \mathcal L(\cZ)$ are self-adjoint operators. In this case we have
$$ P_{\cX} \, (A\, P_{\cY} -P_{\cY}\, A\, P_{\cY})  = \begin{pmatrix}  P^{\cX} \, (A_\cZ \, P^{\cY} -P^{\cY}\, A_\cZ\, P^{\cY}) & 0 \\ 0 & 0\end{pmatrix}\,.$$ 
On the other hand, a simple calculation show that 
$$(s(P_{\cX}R_{Y}),0_{d-k})=s(P_{\cX} \, (A\, P_{\cY} -P_{\cY}\, A\, P_{\cY}) )\in(\R^d_{\geq 0})\da\,.$$ 
Hence, $(s(P_{\cX}R_{Y}),0_{p-k})= s(P^{\cX} \, (A_\cZ \, P^{\cY} -P^{\cY}\, A_\cZ\, P^{\cY}))=s(P^{\cX} (I_\cZ- P^ \cY) \, A_\cZ\, P^{\cY})$. Thus, we can assume further that $\C^d=\cZ=\cX+\cY$ and show that 
 \begin{equation}\label{eqn a probar pro1}
(s(P_{\cX}\, R_Y),0_{d-k})=s(P_{\cX}\, (P_{\cY^\perp} A\, P_\cY))\prec_w (\spr(A) \, \sin(\Theta(\cX,\cY)),0_{d-k})\, .
\end{equation}
Now using item 3 of Theorem \ref{theorem ag} (Lidskii's multiplicative property), 
\begin{equation}\label{eqn19}
 s(P_{\cX}P_{\cY\orto}AP_{\cY})=s(P_{\cX}P_{\cY\orto} P_{\cY\orto} A P_{\cY})\prec_w s(P_{\cX}P_{\cY\orto})\, s(P_{\cY\orto}AP_{\cY}).
 \end{equation}
First noticing that by Theorem \ref{theoremr Fem4.6}, we have that $s(P_{\cX}P_{\cY\orto})=(\sin(\Theta(\cX,\cY)),0_{d-k})$.
On the other hand, consider the matrix representation induced by the decomposition $\C^d=\cY\oplus \cY^\perp$:
\begin{equation}\label{eqn13}
A=\begin{pmatrix}
A_{11}&A_{21}^*\\
A_{21}&A_{22}
\end{pmatrix} \ \text{ and set } \ A_1:=\begin{pmatrix}
A_{11}&0\\
0&A_{22}
\end{pmatrix} \ \text{ , } \ A_2:=\begin{pmatrix}
0&A_{21}^*\\
A_{21}&0
\end{pmatrix}\,.
\end{equation}
Then, we have that $A=A_1+A_2$.  Now, $A_1$ is a pinching of $A$ (associated with the system of projections $\{P_{\cY}\coma P_{\cY\orto}\}$) 
so $\lambda(A_1)\prec \lambda(A)$ so then
\begin{equation}\label{eqn14}-\lambda^{\uparrow}(A_1)\prec -\lambda^{\uparrow}(A)\,.
\end{equation}  
Using Lidskii's additive property for $A_2=A-A_1$ (see item 1 in Theorem \ref{theorem ah})
 \begin{equation}\label{eqn15}
 \lambda(A_2)\prec\lambda(A)-\lambda^{\uparrow}(A_1)\,.
\end{equation}
Combining \eqref{eqn14} and \eqref{eqn15}, we obtain
\begin{equation}\label{eqn16}
\lambda(A_2)\prec \lambda(A)-\lambda^{\uparrow}(A)=\spr(A)\in \R^{d}\,.
\end{equation}
By Proposition \ref{hat trick como en el futbol}, 
we get that $\lambda(A_2)=(s(A_{21}),-s(A_{21}^*))\da$; 
in particular, $s(A_{21})=(\la_i(A_2))_{i\in \I_k}$. 
Now, $s(P_{\cY\orto}AP_{\cY})=(s(A_{21}),0_{d-k})$; thus, we see that
\begin{equation}\label{eqn17}
s(P_{\cY\orto}AP_{\cY})=(s(A_{21}),0_{d-k})=((\la_i(A_2))_{i\in \I_k},0_{d-k})\prec_w ((\spr_i(A))_{i\in\I_k},0_{d-k})\, ,
\end{equation} where $\spr(A)=(\spr_i(A))_{i\in\I_d}$. Using \eqref{eqn19} and \eqref{eqn17} together with Lemma \ref{lemma submaj props1} we finally get that 
$$  s(P_{\cX}P_{\cY\orto}AP_{\cY})\prec_w (\spr(A) \, \sin(\Theta(\cX,\cY)),0_{d-k})\in (\R_{\geq 0}^d)\da\,.$$
Now the result follows from the last submajorization relation, by considering the first $k$ entries of both vectors.
\end{proof}

\medskip
\noi
{\scshape Proposition \ref{proSin}.}
Let $A\in\H(d)$, $\cX,\cY\subset \C^d$ subspaces with $\dim(\cX)=\dim(\cY)=k$. Assume that  $\cX$ is $A$-invariant.
Then, 
\begin{equation}\label{eqn21 bis}
s(P_{\cX}R_{Y})\prec_w 2\ (\la_i(A_{\cX+\cY})- \lambda_{\min}(A_{\cX+\cY}))_{i\in\I_k} \ \sin^2(\Theta(\cX,\cY)).
\end{equation}
\begin{proof}
Arguing as in the proof of Proposition \ref{prospread}, we can assume further that $\C^d=\cX+\cY$. With this assumption, we consider first
the case where $A\in\matpos$ and show that 
\begin{equation}\label{eqn21ap}
s(P_{\cX}R_{Y})\prec_w 2\,(\la_i(A))_{i\in\I_k}\,\sin^2(\Theta(\cX,\cY))\, .
\end{equation}
Indeed, the $A$-invariance of $\cX$, allows us to write $A=P_{\cX}AP_{\cX}+P_{\cX\orto}AP_{\cX\orto}$. With 
this decomposition in mind using the fact that $(s(P_{\cX}R_{Y}),0_{d-k})=s(P_{\cX}P_{\cY\orto}AP_{\cY})$, we have that
\begin{align}\label{eqn22}
\nonumber s(P_{\cX}P_{\cY\orto}A\,P_{\cY})&=s(P_{\cX}\pyorto P_{\cX}A \,P_{\cX} P_{\cY}+P_{\cX}\pyorto \pxorto A \,\pxorto P_{\cY})
\\& \nonumber\prec_w 
 s(P_{\cX}\pyorto P_{\cX}A\,P_{\cX}P_{\cY})+s(P_{\cX}\pyorto A \,\pxorto P_{\cY})
 \igdef M \ .
\end{align}
Using item 3 of Theorem \ref{theorem ag} (Lidskii's multiplicative property), 
the fact that $0_d\leq s(P_{\cX}\, P_{\cY})\leq \uno_d$ and Theorem \ref{theoremr Fem4.6}, we get 
\begin{align}
\nonumber 
M & 
\prec_w\nonumber  
s(P_{\cX}\pyorto P_{\cX})\, s(A)+
s(P_{\cX}\pyorto)\, s(A) \, s(\pxorto P_{\cY})
\\& \nonumber \prec_w 
2\,\la(A)\,(\sin^2(\Theta(\cX,\cY)),0_{d-k})\in(\R_{\geq 0}^d)\da\,,
\end{align}
since $A\in\matpos$ is positive semi-definite. The result now follows from the previous facts.

\pausa In general, for $A\in\H(d)$ consider the auxiliary matrix $\tilde A=A-\la_{\min(A)}\, I\in\matpos$.
Notice that 
$$
R_Y(\tilde A)= \tilde A\, Y- Y(Y^*\tilde A\, Y)=  A\, Y- Y(Y^* A\, Y) =R_Y 
\,,  
$$ and $\la(\tilde A)=\la(A)-\la_{\min(A)}\, \uno_d$. The result now follows from these facts and from 
\eqref{eqn21ap} applied to 
$\tilde A$.
\end{proof}

\pausa
{\bf Acknowledgment.} We would like to thank the reviewers of the manuscript for providing 
several useful comments that helped us improve the presentation of the results herein.

\end{document}